\documentclass[draft]{asl}
\usepackage{braket}
\usepackage{mathrsfs, color}

\numberwithin{equation}{section} 
\numberwithin{figure}{section} 

\newtheorem{thm}{Theorem}
\newtheorem{lem}[thm]{Lemma}
\newtheorem{prop}[thm]{Proposition}

\newtheorem{fact}[thm]{Fact}

\theoremstyle{definition}

\newtheorem{defin}[thm]{Definition}

\newcommand{\La}{\mathscr{L}}
\newcommand{\Aut}{\mathop\mathrm{Aut}}
\newcommand{\Orb}{\mathrm{Orb}}
\newcommand{\Inv}{\mathop\mathrm{Inv}}

\newcommand\N{\mathbb{N}}

\let\oldmarginpar\marginpar
\renewcommand{\marginpar}[1]{\-\oldmarginpar[\raggedleft\footnotesize #1]%
{\raggedright\footnotesize #1}}

\begin{document}

\title{Non-permutation invariant Borel quantifiers}

\author{Fredrik Engstr\"om}
\revauthor{Engstr\"om, Fredrik}
\address{Department of Philosophy, Linguistics and Theory of Science \\ University of Gothenburg\\ 
  Box 200, 405 30 G\"oteborg, Sweden}
\email{fredrik.engstrom@gu.se}
\thanks{Part of the work in this paper was done while visiting the Institut Mittag-Leffler. The authors would like thank the Institut Mittag-Leffler for support.}
\thanks{First author partially supported by the EUROCORE LogICCC LINT program and the Swedish Research Council.}

\author{Philipp Schlicht}
\revauthor{Schlicht, Philipp}
\address{Mathematisches Institut \\
Universit\"at Bonn \\
Endenicher Allee 60, 53115 Bonn, Germany }
\email{schlicht@math.uni-bonn.de}
\thanks{Second author partially supported by an exchange grant from the European Science Foundation.}

\begin{abstract} 
Every permutation invariant Borel subset of the space of countable structures is definable in $\La_{\omega_1\omega}$ by a theorem of Lopez-Escobar. We prove variants of this theorem relative to fixed relations and fixed non-permutation invariant quantifiers. Moreover we show that for every closed subgroup $G$ of the symmetric group $S_{\infty}$, there is a closed binary quantifier $Q$ such that the $G$-invariant subsets of the space of countable structures are exactly the $\La_{\omega_1\omega}(Q)$-definable sets. \end{abstract} 

\maketitle

\section{Introduction}

Countable models in a given countable relational signature $\tau$ can be represented as elements of the logic space $$X_{\tau}=\prod_{R\in \tau} 2^{\mathbb{N}^{a(R)}}$$ where $a(R)$ denotes the arity of the relation $R$. For example, the set of elements of the logic space for a binary relation representing linear orders is a closed set. 

The Lopez-Escobar theorem is an easy consequence due to Scott of the interpolation theorem \cite{Lopez-Escobar:65} for $\La_{\omega_1\omega}$. The interpolation theorem states that if $\varphi$ is an $\La_{\omega_1\omega}$-formula in the signature $\sigma$ and $\psi$ is an $\La_{\omega_1\omega}$-formula in the signature $\tau$ such that $\varphi\rightarrow\psi$ holds in all countable models, then there is an  $\La_{\omega_1\omega}$ interpolant $\theta$ in the signature $\sigma\cap\tau$ such that $\varphi\rightarrow\theta$ and $\theta\rightarrow\psi$ hold in all countable models. 

The Lopez-Escobar theorem \cite[theorem 16.8]{Kechris:95} states that any invariant Borel subset of the logic space is defined by a formula in $\La_{\omega_{1}\omega}$. To derive this from the interpolation theorem, note that every Borel set is defined by an $\La_{\omega_{1}\omega}$-formula in a sequence of parameters $n_i\in\mathbb{N}$. If you replace each $n_i$ by a constant $c_i$ or $d_i$ and use the fact that the set is permutation invariant, it follows that there is an $\La_{\omega_1\omega}$ interpolant without parameters or constants defining the set. Vaught \cite{Vaught:74} found a different proof which has the advantage that it generalizes to the logic space for structures of higher cardinalities. We will generalize Vaught's proof to sets of countable structures invariant under the action of a closed subgroup of the permutation group of the natural numbers. 

Let $G$ be the group of permutations fixing a countable family of relations and constants in the natural numbers. In section 2 we show that every $G$-invariant set is definable from these relations and constants. 

A generalized quantifier of type $\langle k\rangle$ on the natural numbers is a subset of $2^{\mathbb{N}^k}$. We will freely identify subsets of $\mathbb{N}^k$ and their characteristic functions. We consider the logic $\La_{\omega_1\omega}(Q)$. This is $\La_{\omega_1\omega}$ augmented by the quantifier $Q$ where the formula $Qx\varphi(x)$ has the fixed interpretation $\{x\in\mathbb{N}^k:\varphi(x)\}\in Q$. We study non-permutation invariant generalized quantifiers on the natural numbers and prove a variant of the Lopez-Escobar theorem for a subclass of the quantifiers which are closed and downwards closed. 

Moreover for every closed subgroup $G$ of the symmetric group $S_{\infty}$, there is a closed binary quantifier $Q$ such that the $G$-invariant subsets of the space of countable structures are exactly the $\La_{\omega_1\omega}(Q)$-definable sets. 

In section 3 we show that there is a version of the Lopez-Escobar theorem for clopen quantifiers and for finite boolean combinations of principal quantifiers. In section 4 we generalize some of the results to the logic space for structures of size $\kappa$ for uncountable cardinals $\kappa$ with $\kappa^{<\kappa}=\kappa$.

\section{Variants of the Lopez-Escobar theorem}

Let 
$$X_\tau = \prod_{R \in \tau} 2^{\N^{a(R)}}$$ 
denote the logic space on $\N$ for a relational signature $\tau$, where $a(R)$ is the arity of the relation $R$. The space is equipped with the product topology. If $\mathcal{F}$ is a sequence of relations on $\mathbb{N}$, then the logic $\La_{\omega_1 \omega}(\mathcal{F})$ has a symbol for each relation in $\mathcal{F}$ with fixed interpretation as this relation.

\subsection{Variants relative to relations} 

We prove a version of the Lopez-Escobar theorem for closed subgroups of the permutation group $S_{\infty}$ of the natural numbers. Recall the standard 

\begin{fact} The closed subgroups of $S_{\infty}$ are exactly the automorphism groups of countable relational structures. 
\end{fact}

See for example \cite[theorem 2.4.4]{Gao:09} for a proof. 

\begin{defin} Suppose $G\leq S_{\infty}$ is a subgroup. The $G$-orbit of $a\in \N^{<\omega}$ is defined as $\Orb_G(a)=\{g(a):g\in G\}$. 
\end{defin} 

When $G$ is understood from the context we write $\Orb(a)$ for the $G$-orbit of $a$.

\begin{prop} 
Suppose $G\leq S_{\infty}$ is closed and $\mathcal{F}$ is the family of orbits of $G$. Then every $G$-invariant Borel subset of $X_{\tau}$ is definable in $\La_{\omega_{1}\omega}(\mathcal{F})$. 
\end{prop} 

\begin{proof} 
The proof is very similar to Vaught's proof  \cite{Vaught:74}. We follow the proof of \cite[theorem 16.8]{Kechris:95} and replace the set of injections $k\rightarrow \mathbb{N}$ with the orbit of $\langle0,1,..,k-1\rangle$. Note that the Baire category theorem for $G$ holds since $G$ is closed in $S_{\infty}$. By induction on the Borel rank, there is for every Borel set $A\subseteq X_{\tau}$ and every $k\in \N$ an $\La_{\omega_{1}\omega}(\mathcal {F})$-formula $\varphi_k$ such that $\varphi_k(x,a)$ holds for  $\langle x,a\rangle\in X_{\tau}\times \N^k$ if and only if $g(x)\in A$ for comeager many $g\in G$ with $a\subseteq g^{-1}$. 
\end{proof} 

Note that every $G_{\delta}$ subgroup of $S_{\infty}$ is closed \cite[proposition 1.2.1]{Becker.Kechris:96}. However, Proposition 3 is false for some $F_\sigma$ subgroups. We write $A=^*B$ if $A\triangle B$ is finite and $A\subseteq^* B$ if $A-B$ is finite. Suppose $A\subseteq \N$ is infinite and coinfinite and let $G=\set{g \in S_\infty: g(A)=^*A}$. Then $G$ is $F_{\sigma}$ and has the same orbits as $S_\infty$. However, the set $Q=\{X:A\subseteq^*X\}$ is $G$-invariant, but not $S_{\infty}$-invariant and hence not definable from the orbits of $G$. 

\begin{defin} Suppose $G\leq S_{\infty}$. The orbit equivalence relation on $\N^{<\omega}$ is defined as $E_G=\{\langle a,b\rangle\in\N^{<\omega}\times \N^{<\omega}: \exists g\in G (g(a)=b)\}$. 
\end{defin} 

The orbit equivalence relation may contain much less information than the family of orbits. For example, $E_{\{id_{\mathbb{N}}\}}$ is definable in $\La_{\omega_1\omega}$. Hence none of the orbits of $\{id_{\mathbb{N}}\}$ is definable from $E_{\{id_{\mathbb{N}}\}}$. 

As a corollary to lemma 3 we obtain a variant of Scott sentences for countable structures. 

\begin{prop} 
Suppose $G\leq S_{\infty}$ is closed and $\mathcal{F}$ is the family of $G$-orbits. There is for each $M \in X_\tau$ an $\La_{\omega_1\omega}(\mathcal{F})$-sentence $\varphi^G_M$ with $M\vDash\varphi^G_M$ and the property that $\varphi^G_M=\varphi^G_N$ if and only of there is $g\in G$ with $g(M)=N$. 
\end{prop} 

\begin{proof} 
The orbit $\Orb(M)=\{g(M):g\in G\}$ is Borel \cite[Theorem 3.3.2]{Gao:09}. 
\end{proof} 

When $G$ is the symmetric group we use the standard notation $\varphi_M=\varphi^{S_\infty}_M$. We give a version of the Lopez-Escobar theorem relative to a family of relations. 

\begin{prop} Suppose $\mathcal{F}=\langle R_i:i<\omega\rangle$ is a family of relations on $\mathbb{N}$ in the signature $\tau$. Then every $\Aut(\mathcal{F})$-invariant Borel subset of $X_\tau$ is definable in $\La_{\omega_1\omega}(\mathcal{F})$.
\end{prop} 

\begin{proof} 

It is sufficient to show that each orbit of $\Aut(\mathcal{F})$ in $\mathbb{N}^{<\omega}$ is definable in $\La_{\omega_1\omega}(\mathcal{F})$. Then all $\Aut(\mathcal{F})$-invariant Borel sets are definable in $\La_{\omega_1\omega}(\mathcal{F})$ by Proposition 3. Note that $\Aut(\mathcal{F})$ is closed. 

Let $M_a=\langle\mathbb{N},\langle R_i:i<\omega\rangle,a\rangle$ for $a\in\mathbb{N}^{<\omega}$. Then $a\in \Orb(b)$ if and only if the structures $M_a$ and $M_b$ are isomorphic via a permutation in $G$ if and only if $M_b\vDash \varphi_{M_a}$. Hence $\varphi_{M_a}$ defines the orbit of $a$. 
\end{proof}

\subsection{Variants relative to quantifiers} 

Let $Q$ be a quantifier on $\mathbb{N}^k$, i.e. a quantifier of type $\langle k\rangle$ on the natural numbers. We have to look at two kinds of definability. 

\begin{defin} A set $A\subseteq X_{\tau}$ is definable in $\La_{\omega_1\omega}(Q)$ if there is an $\La_{\omega_1\omega}(Q)$-formula $\varphi$ such that $M\in A$ if and only if $M \vDash \varphi$ for all $M\in X_\tau$. 
\end{defin} 

\begin{defin} A set $A\subseteq\mathbb{N}^k$ is definable in $\La_{\omega_1\omega}(Q)$ if there is an $\La_{\omega_1\omega}(Q)$-formula $\varphi$ such that $n\in A$ if and only if $\N \vDash\varphi(n)$. 
\end{defin} 

We say that a permutation $f$ fixes $Q$ if $A\in Q$ is equivalent to $f(A)\in Q$ for all $A\subseteq \N^k$. $\Aut(Q)$ is the group of permutations fixing $Q$. A set $A\subseteq X_{\tau}$ is called $G$-invariant for $G\leq S_{\infty}$ if $g(A)=A$ for all $g\in G$.  

Proposition 3 implies 

\begin{prop} 
Suppose $Q\subseteq 2^{\mathbb{N}^k}$ is a Borel quantifier with closed automorphism group. Suppose the orbits of $\Aut(Q)$ are definable in $\La_{\omega_{1}\omega}(Q)$. Then a subset of $X_{\tau}$ is Borel and $\Aut(Q)$-invariant if and only if it is definable in $\La_{\omega_{1}\omega}(Q)$.
\end{prop} 

\begin{proof} 
We are left to show that every $\La_{\omega_{1}\omega}(Q)$-definable set $A\subseteq X_{\tau}$ is Borel. Note that the Borel sets are exactly the sets definable in $\La_{\omega_1\omega}$ in a sequence of parameters $m_i\in\mathbb{N}$. Suppose $Q$ is defined by an $\La_{\omega_{1}\omega}$-formula $\varphi$ in the parameters $m_i$.   

We want to show by induction on $\psi$ that if $A$ is defined by the $\La_{\omega_{1}\omega}(Q)$-formula $\psi(x,\vec{n})$ with $\vec{n}=\langle n_i:i<\omega\rangle$, then $\psi$ is equivalent to an $\La_{\omega_{1}\omega}$-formula in a sequence of natural parameters. Let $\psi=Qx\chi(x,\vec{n})$ where $\chi$ is an $\La_{\omega_{1}\omega}$-formula. 

Suppose $M$ is a countable structure in the signature $\tau$. Then $M\in A$ if and only if $\set{x:M\vDash \chi(x,\vec{n}}\in Q$ if and only if $\langle\mathbb{N}^k, \set{x:M\vDash\chi(x,\vec{n})}\rangle\vDash\varphi$. Hence $A$ is Borel. 
\end{proof} 

The assumption that $\Aut(Q)$ is closed is essential. We write $A=^*B$ if $A\triangle B$ is finite. Let $Q=\{X:X=^*A\}$ where $A\subseteq \mathbb{N}$ is infinite and coinfinite. It follows by induction on $\varphi$ that any  $\La_{\omega_1\omega}(Q)$-formula of the form $Qx\varphi(x,a)$ is false in $\langle\mathbb{N},X\rangle$ for all $X\neq^*A,\neg A$. If $\langle\N,X\rangle\vDash\varphi$ for some infinite and co-infinite $X\neq^*A,\neg A$ where $\varphi$ is an $\La_{\omega_1\omega}(Q)$-formula, this implies that $\varphi$ holds in every structure $\langle\N,Y\rangle$ with $Y\neq^*A,\neg A$ infinite and co-infinite. Hence the set $\{X:A\subseteq^*X\}$ is $\Aut(Q)$-invariant but not  $\La_{\omega_1\omega}(Q)$-definable. 

\begin{defin} Suppose $Q\subseteq 2^{\mathbb{N}^k}$ is closed and downwards closed, i.e. closed under subsets. A function $p:n\rightarrow\mathbb{N}$ is compatible with $Q$ if and only if for every $A\subseteq n^k$, $A\in Q$ if and only if $p(A)\in Q$. 
\end{defin} 

Note that for downwards closed $Q$ and $A\subseteq n^k$, $A\in Q$ if and only if $A$ extends to an element of $Q$, i.e. if there is $B\subseteq \mathbb{N}^k$ in $Q$ with $A=B\cap n^k$. Hence $p:n\rightarrow\mathbb{N}$ is compatible with $Q$ if and only if for every function $g:n^k\rightarrow \{0,1\}$, $g$ extends to an element of $Q$ if and only if $p(g)=g\circ p^{-1}$ extends to an element of $Q$.

\begin{defin} A quantifier $Q$ is good if it is closed, downwards closed, and any finite injection $p:n\rightarrow\mathbb{N}$ compatible with $Q$ extends to a permutation $f:\mathbb{N}\rightarrow\mathbb{N}$ leaving $Q$ invariant.
\end{defin} 

The information about tuples of natural numbers encoded in a good quantifier $Q$ is definable in $\La_{\omega\omega}(Q)$.

\begin{prop} Suppose $Q$ is good. Then the orbits of $\Aut(Q)$ are definable in $\La_{\omega\omega}(Q)$.
\end{prop}

\begin{proof} $a=\langle a_i:i<n\rangle$ is in the orbit of $\langle0,..,n-1\rangle$ if and only if the map $n\rightarrow\mathbb{N}$ mapping $i$ to $a_i$ is compatible with $Q$. This is expressible by the conjunction of $Qx\bigvee_{i\in I}x=\langle a_{i(0)},..,a_{i(k-1)}\rangle$ for all $I\subseteq n^k$ with $I\in Q$ together with the conjunction of $\neg Qx\bigvee_{i\in I}x=\langle a_{i(0)},..,a_{i(k-1)}\rangle$ for all $I\subseteq n^k$ with $I\notin Q$. 
\end{proof} 

Note that the automorphism group of a closed quantifier is closed. 

\begin{prop} 
Suppose $G$ is a closed subgroup of $S_{\infty}$. There is a good binary quantifier $Q_G$ with $G=\Aut(Q_G)$. 
\end{prop} 

\begin{proof} 
Let $P$ be the downward closure of $$\bigcup_{k\in\mathbb{N}} \Orb(\{\langle0,0\rangle,\langle0,1\rangle,\langle1,2\rangle,\ldots,\langle k-1,k\rangle\})$$
Then $P$ is $G$-invariant, so its closure $Q$ is $G$-invariant as well. 

Suppose $p:k\rightarrow\mathbb{N}$ is a finite injection compatible with $Q$. Then $s=\{\langle p(0),p(0)\rangle,\langle p(0),p(1)\rangle,..,\langle p(k-2),p(k-1)\rangle\}\in Q$. Let $\langle a^n: n<\omega\rangle$ be a sequence in $P$ converging to $s$. Then $a^n$ eventually contains a set of the form $\{\langle a^n_0,a^n_0\rangle,..,\langle a^n_{k-2},a^n_{k-1}\rangle\}$ and the eventual value of $a^n_i$ is $p(i)$ for all $i<k$. Hence $s\in P$ and $p$ can be extended to a permutation in $G$. 

We claim that $G=\Aut(Q)$. Let $g\in \Aut(Q)$. Then for every $m$ we have $\{\langle g(0),g(0)\rangle,\langle g(0),g(1)\rangle,..,\langle g(m-1),g(m)\rangle\}\in Q$. This is in fact an element of $P$ by the same argument as in the last paragraph. So there are permutations $h_m\in G$ for each $m$ with $g(i)=h_m(i)$ for all $i\leq m$. Since $h_m\rightarrow g$, this implies $g\in G$. 
\end{proof} 

Hence there is a correspondence between the closed subgroups of $S_{\infty}$ and good binary quantifiers $Q$. Let $\Inv(G)$ denote the family of closed $G$-invariant subsets of $2^{\mathbb{N}^2}$. Then $$\Aut(\Inv(G))=G$$ for every closed subgroup $G\leq S_{\infty}$.

\section{More quantifiers} 

We show that some other types of quantifiers have similar properties as good quantifiers, i.e. their automorphism group is closed and each orbit is definable from the quantifier. For these quantifiers there is a version of the Lopez-Escobar theorem. 

However, the set of quantifiers with these properties is not closed under unions or intersections. To see this, we consider quantifiers of the following form. 

\begin{defin} A principal quantifier is of the form 
\begin{itemize} 
\item $Q_A=\set{X\subseteq \N^k: A\subseteq X}$ or 
\item $Q^A=\set{X\subseteq \N^k: X\subseteq A}$ 
\end{itemize} 
where $A$ is a subset of $\mathbb{N}^k$. 
\end{defin} 

The automorphism group $\Aut(Q_A)=\Aut(Q^A)=\Aut(Q_A\cap Q^A)=\Aut(A)$ of a principal quantifier is closed and its orbits are definable in $\La_{\omega \omega}(Q_A)$. For $A\subseteq\mathbb{N}$ this is true since $m\in A$ if and only iff $\neg Q_An(m\neq n)$ holds, and for $A\subseteq\mathbb{N}^k$ this is shown in section 3.2. 

Let's fix some infinite and co-infinite set $A\subseteq \N$ and let $Q=\{A\}$, so that $\Aut(Q)=\Aut(A)$. A sentence $Qx\varphi(x,a)$ is false whenever $\varphi$ is a $\La_{\omega_{1}\omega}$-formula and $a\in\mathbb{N}^{<\omega}$, since the set $\{n:\varphi(n,a)\}$ is invariant under permutations fixing $a$. Thus any subset of $\mathbb{N}$  defined by an $\La_{\omega_{1}\omega}(Q)$-formula with parameters in $\{0,..,t\}$ is either a subset of $\{0,..,t\}$ or includes $\N -\{0,..,t\}$ by induction on the formulas. Hence the orbits of $\Aut(Q)$ are not definable in $\La_{\omega_{1}\omega}(Q)$.

\subsection{Clopen quantifiers}

Suppose $Q$ is a quantifier of type $\langle k\rangle$. Note that the automorphism group $\Aut(Q)$ of any closed quantifier $Q$ is closed. 

We say that a set $S \subseteq \N^k$ supports $Q$ if $A\in Q$ if and only if $B\in Q$ for all $A,B\subseteq \N$ with $A\cap S=B\cap S$. 

\begin{defin}
A minimal set $S \subseteq \N^k$ supporting $Q$ is called a support of $Q$. 
\end{defin}
 
\begin{lem}
Every closed quantifier $Q$ on $\mathbb{N}^k$ has a unique support. 
\end{lem}

\begin{proof}
Easily the set of $S \subseteq \N^k$ which support $Q$ is closed under finite intersections. Suppose $S_n$ supports $Q$ for each $n\in \N$ and $S_n\subseteq S_m$ for $m\leq n$. Let $S=\bigcap_{n\in \N}S_n$. Suppose $A\in Q$ and $A\cap S=B\cap S$. Now $A_n=(A\cap S_n)\cup(B-S_n)\in Q$ for each $n$ since $A\in Q$. Then $B\in Q$ since $B$ is the limit of the sets $A_n$ and $Q$ is closed. 

Let $\N^k=\{a_n:n\in \N\}$. The support of $Q$ is the intersection of the sets $A_n$ where $A_0=\N^k$ and $A_{n+1}=A_n-\{a_n\}$ if this set supports $Q$ and $A_{n+1}=A_n$ otherwise. 
\end{proof}

Note that the set of finite subsets of $\N^k$ is supported by $\N^k-\{a\}$ for every $a\in \N^k$, so it does not have a support.

\begin{lem} The support of any clopen quantifier $Q$ on $\mathbb{N}^k$ is definable in $\La_{\omega\omega}(Q)$. 
\end{lem}

\begin{proof}
Note that a quantifier is clopen if and only if it has finite support. Suppose the support of $Q$ is contained in $\{0,..,t-1\}^k$ and let $r=t^k$. 

Let $R^{l,m}$ be the set of tuples $\bar{a}\smallfrown \bar{b}$ with $\bar{a}\in(\N^k)^l$ and $\bar{b}\in(\N^k)^m$ such that the finite partial function mapping each $a_{i}\in\N^{k}$ to $0$ and each $b_{j}\in\N^{k}$ to $1$ can be extended to the characteristic function of an element of $Q$. Then $\bar{a}\smallfrown \bar{b}\in R^{l,m}$ if and only if there is a tuple $\bar{c}\in(\N^k)^r$ so that 
$$Qx(\bigwedge_{i<l}x\neq b_{i} \land (\bigvee_{i<m}x=a_{i} \vee \bigvee_{i<r}x=c_{i}))$$ 
holds. Hence $R^{l,m}$ is definable in $\La_{\omega \omega}(Q)$. 

Then $a\in \N^k$ is in the support of $Q$ if and only if there are tuples $\bar{b}\in(\N^k)^r$ and $\bar{c}\in(\N^k)^r$ such that $R^{r+1,r}(\langle a\rangle\smallfrown\bar{b}\smallfrown\bar{c})$ and $R^{r,r+1}(\bar{b}\smallfrown\bar{c}\smallfrown\langle a\rangle)$ don't have the same truth value. 
\end{proof} 

\begin{prop} The orbits of the automorphism group $\Aut(Q)$ of any clopen quantifier $Q$ on $\mathbb{N}^k$ are definable in $\La_{\omega\omega}(Q)$. 
\end{prop}

\begin{proof}
Suppose the support $S$ of $Q$ is contained in $\{0,..,t-1\}^k$. We claim that $\langle a_0,..,a_n\rangle$ is in the orbit of $\langle0,..,n\rangle$ if and only if there is an extension $\langle a_0,..,a_{n+t}\rangle$ with $S\subseteq\{a_0,..,a_{n+t}\}$ such that for the finite partial map $f$ with $f(i)=a_i$ for $i\leq n+t$ 
\begin{itemize} 
\item $f$ is injective and 
\item $f$ and $f^{-1}$ preserve $S$ and $R^{j,l}$ for all $j,l$ with $j+l\leq n+t$. 
\end{itemize} 

Suppose these conditions hold. Let $g$ be any permutation of $\N$ extending $f$. We have $R \in Q$ if and only if $R \cap S$ extends to a relation in $Q$, for any relation $R\subseteq \N^k$. But this holds if and only if $g(R)\cap S$ extends to a relation in $Q$, since $g$ preserves $S$ and $R^{j,l}$. Hence $g\in\Aut(Q)$. 
\end{proof} 

The orbits of $\Aut(\mathcal{F})$ are $\La_{\omega_{1}\omega}(\mathcal{F})$-definable and $\Aut(\mathcal {F})$ is closed for any sequence $\mathcal{F}$ of clopen quantifiers by a slight variation of the previous proof.

\subsection{Combinations of principal quantifiers}

We show that if $Q$ is a finite boolean combination of principal quantifiers $Q_{A_k}$, then its automorphism group is closed and each orbit in $\mathbb{N}^{<\omega}$ is definable in $\La_{\omega_{1}\omega}(Q)$. 

Suppose $\langle A_k:k<n\rangle$ is a partition of $\mathbb{N}^d$ with $d<\omega$ and 
\[Q=\bigcup_{i}\bigcap_{k<n}Q_{A_k}^{s_i(k)}\]
with $s_{i}\in\{1,-1\}^{n}$ for $i<m$, where $Q_{A_k}^1=Q_{A_k}$ and $Q_{A_k}^{-1}=\neg Q_{A_k}$. We can assume that $n$ is minimal with these properties. In this situation we write $Q=\langle A_k, s_i\rangle=\langle A_k, s_i:k<n, i<m\rangle$. 

We say that a tuple $\bar{a}\in (\mathbb{N}^d)^{<\omega}$ occurs positively (negatively) in $Q=\langle A_k, s_i\rangle$ if there is some $i$ such that there is $j$ with $a_j\in A_k$ if and only if $s_i(k)=1$ ($s_i(k)=-1$). A tuple $\bar{a}$ occurs negatively if and only if $\psi(\bar{a}):= Qx\bigwedge_j (x\neq a_{j})$ holds. 

\begin{lem} If $Q=\langle A_k, s_i:k<n, i<m\rangle$, then there is an $\La_{\omega\omega}(Q)$-formula $\chi$ with $\chi(a,b)$ if and only if $a,b\in A_{k}$ for some $k<n$. 
\end{lem} 

\begin{proof} Let $\chi(a,b)$ state that $\psi(\langle a\rangle\smallfrown\bar{c})$, $\psi(\langle b\rangle\smallfrown\bar{c})$, and $\psi(\langle a,b\rangle\smallfrown\bar{c})$ have equal truth values for all tuples $\bar{c}\in(\mathbb{N}^d)^n$. If $a,b\in A_{k}$ for some $k$, then $\chi(a,b)$ holds.

Suppose $a\in A_{0}$, $b\in A_{1}$, and $\chi(a,b)$ holds. Suppose $Q$ is the union of the sets 
\[\bigcup_{t\in T}(\neg Q_{A_{0}}\cap\neg Q_{A_{1}}\cap\bigcap_{j\geq2}Q_{A_{j}}^{t(j)})\] 
\[\bigcup_{u\in U}(Q_{A_{0}}\cap\neg Q_{A_{1}}\cap\bigcap_{j\geq2}Q_{A_{j}}^{u(j)})\] 
\[\bigcup_{v\in V}(\neg Q_{A_{0}}\cap Q_{A_{1}}\cap\bigcap_{j\geq2}Q_{A_{j}}^{v(j)})\] 
\[\bigcup_{w\in W}(Q_{A_{0}}\cap Q_{A_{1}}\cap\bigcap_{j\geq2}Q_{A_{j}}^{w(j)})\] 
We claim that $T=U=V$. To prove $T\subseteq U$, suppose $t\in T$ and pick $\bar{d}$ so that there is exactly one $d_{k}\in A_{k}$ for each $k\geq 2$ with $t(k)=1$, so that $\psi(\langle a,b\rangle\smallfrown\bar{d})$ holds. Then $\psi(\langle b\rangle\smallfrown\bar{d})$ holds and hence $t\in U$. The other cases are analogous. 
This shows that $n$ is not minimal, since
\[(\neg Q_{A_{0}}\cap\neg Q_{A_{1}})\cup(Q_{A_{0}}\cap\neg Q_{A_{1}})\cup(\neg Q_{A_{0}}\cap Q_{A_{1}})\]
 can be replaced by $\neg(Q_{A_{0}}\cap Q_{A_{1}})=\neg Q_{A_{0}\cup A_{1}}$. 
\end{proof} 

Note that the assumption that $n$ is minimal is essential here, since otherwise the proof does not even work for quantifiers of the form $Q=Q_{A}\cap Q_{B}$. 

\begin{lem} If $Q=\langle A_k, s_i:k<n, i<m\rangle$, then for each $j\leq n$ there is an $\La_{\omega\omega}(Q)$-formula $\theta_{j}$ such that $\theta_{j}(\bar{a},\bar{b})$ holds if and only if 
\begin{itemize} 
\item $\bar{a}$ occurs positively and has length $j$, 
\item $\bar{b}$ occurs negatively and has length $n-j$, and 
\item all elements of $\bar{a}\smallfrown\bar{b}$ are in different $A_{k}$. 
\end{itemize} 
\end{lem} 

\begin{proof} The formula $\theta_j$ can be expressed by $\chi$ and $\psi$. Note that if $\bar{b}$ occurs negatively, then $\bar{a}$ has to occur positively, given the remaining conditions. 
\end{proof} 

\begin{lem} If $Q=\langle A_k, s_i:k<n, i<m\rangle$, then $g\in\Aut(Q)$ if and only if there are a permutation $p$ of $n$ and a permutation $r$ of $m$ such that $g(A_{k})=A_{p(k)}$ for all $k<n$ and $s_{r(i)}=s_{i}\circ p^{-1}$ for all $i<m$. 
\end{lem} 

\begin{proof} 
If $g\in \Aut(Q)$, then $g$ permutes the $A_k$ by the previous lemma. Let $p:n\rightarrow n$ be this permutation. For each $s:n\rightarrow \{-1,1\}$, there is $i<m$ with $s=s_i$ if and only if there is some $j<m$ with $s\circ p=s_j$. 

Suppose $p$ and $r$ are given and $x\in\bigcap_{k<n}Q_{A_k}^{s_i(k)}$ for some $i<n$. Then $g(x)\in\bigcap_{k<n}Q_{A_{p(k)}}^{s_i(k)}=\bigcap_{k<n}Q_{A_{k}}^{s_{r(i)}}$. 

\end{proof} 

This implies that $\Aut(Q)$ is closed. Suppose $g_k\rightarrow g\in S_{\infty}$ with $g_k\in \Aut(Q)$ for each $k<\omega$ and let $p_k$ be the permutation of $n$ corresponding to $g_k$ in the previous lemma. Then $p_k$ eventually takes a fixed value $p$, hence $g$ is according to $p$. 

Given a tuple $\bar{a}\in(\mathbb{N}^d)^j$, we can find $f:j\rightarrow n$ such that there is a tuple $\bar{c}\in(\mathbb{N}^d)^n$ with 
\begin{itemize} 
\item all $c_i$ are in different $A_k$ and $\bar{c}$ is maximal with this property, and 
\item $a_i$ and $c_{f(i)}$ are in the same $A_k$ for each $i<j$. 
\end{itemize} 

For tuples $\bar{c}$ with this property, let $M_{\bar{a},\bar{c}}=\langle\mathbb{N},\bar{a},\bar{c},\langle A_{p(k)}:k<n\rangle\rangle$, where $p$ is the unique permutation of $n$ such that $c_k\in A_{p(k)}$. Note that the Scott sentence $\varphi_{M_{\bar{a},\bar{c}}}$ of $M_{\bar{a},\bar{c}}$ is equivalent to a sentence in $\La_{\omega_1\omega}(Q)$ with  parameters $\bar{a}$ and $\bar{c}$, since $A_{p(k)}$ is definable from $Q$ and $c_i$. 

\begin{prop} For any finite boolean combination $Q$ of principal quantifiers of the form $Q_{A}$, the orbits of $\Aut(Q)$ are definable in $\La_{\omega_1\omega}(Q)$. 
\end{prop} 

\begin{proof} Let $Q=\langle A_k, s_i:k<n, i<m\rangle$. Suppose $\bar{a}$ is a tuple of length $j$ and $f:j\rightarrow n$ and $\bar{c}$ are as above. 

We claim that $\bar{b}\in \Orb(\bar{a})$ if and only if there is a tuple $\bar{d}\in\mathbb{N}^n$ such that 
\begin{itemize} 
\item all $d_i$ are in different $A_k$ and $\bar{d}$ is maximal with this property, 
\item $b_i$ and $d_{f(i)}$ are in the same $A_k$ for each $i$, 
\item $M_{\bar{b},\bar{d}}\vDash \varphi_{M_{\bar{a},\bar{c}}}$, and 
\item for all $I\subseteq n$, $Qn(\bigwedge_{k\in I}{n\neq d_k})$ holds if and only if $I=\{k<n:s_i(k)=-1\}$ for some $i<n$. 
\end{itemize} 

Suppose these conditions hold for $\bar{b}$ and $\bar{d}$. Since $M_{\bar{b},\bar{d}}$ models $\varphi_{M_{\bar{a},\bar{c}}}$, there is a permutation $g:\mathbb{N}\rightarrow\mathbb{N}$ mapping $\bar{a}$ to $\bar{b}$ and $\bar{c}$ to $\bar{d}$. Let $p:n\rightarrow n$ be the permutation of the indices of $c_i$ induced by this map. Then $g(A_k)=A_{p(k)}$ for each $k<n$. The last condition implies that for every $i<m$ there is some $j<m$ such that $s_i\circ p=s_j$. Hence $g$ preserves $Q$ by the previous lemma. 
\end{proof} 

Note that the proposition is also true for boolean combinations of principal quantifiers $Q^{A_k}$ since $Q^{A}x\varphi (x)$ can be expressed as $Q_{\neg A}x\neg\varphi (x)$. However, the two types of principal quantifiers cannot be mixed by the example at the beginning of section 3. 

By a slight variation of the previous proof we get 

\begin{prop} Suppose $\mathcal{F}=\langle Q_i:i<\omega\rangle$ is a sequence of finite boolean combinations of principal quantifiers $Q_{A_{i,k}}$. Then $\Aut(\mathcal{F})$ is closed and the orbits of $\Aut(\mathcal{F})$ are definable in $\La_{\omega_1\omega}(\mathcal{F})$. 
\end{prop}

\section{Higher cardinalities}

Some of the previous results generalize when $\mathbb{N}$ is replaced with an uncountable cardinal $\kappa$.  Let's always suppose $\kappa^{<\kappa}=\kappa$. 

The logic space
$$X_{\tau}=\prod_{R\in \tau} 2^{\kappa^{a(R)}}$$ 
for a relational signature $\tau$ of size $\leq\kappa$ is equipped with the product topology. The topology on $2^{\kappa^n}$ is given by the basic open sets  $U(s)=\{f\in 2^{\kappa^n}:s\subseteq f\}$ for partial functions $s\in 2^{\kappa^n}$ of size $<\kappa$. Let $S_\kappa$ denote the permutation group of $\kappa$ with the topology from $\kappa^\kappa$. 

The $\kappa$-Borel subsets of $2^{\kappa^n}$ and $X_{\tau}$ are generated from the basic open sets by unions and intersections of length $\kappa$ and complements. A subspace is $\kappa$-Baire if $\bigcap_{\alpha<\kappa}U_{\alpha}$ is dense in the subspace for every sequence $\langle U_{\alpha}:\alpha<\kappa\rangle$ of open dense sets in the subspace. 

A generalized quantifier of type $\langle\alpha\rangle$ on $\kappa$ for $\alpha<\kappa$ is a subset of $2^{\kappa^\alpha}$. 

\begin{lem} 
Suppose $Q$ is a closed quantifier on $\kappa$ of type $\langle\alpha\rangle$ with $\alpha<\kappa$. Then $\Aut(Q)$ is closed in $S_\kappa$. 
\end{lem} 

\begin{proof} 
Suppose $g_\alpha\in\Aut(Q)$ for each $\alpha<\kappa$ and $g_\alpha\rightarrow g\in S_\kappa$. Let $R$ be a relation in $Q$. Then $g_\alpha(R)\rightarrow g(R)$ and hence $g(R)\in Q$. Since $g_\alpha^{-1}\rightarrow g^{-1}$ we have $g\in\Aut(Q)$. 
\end{proof} 

\begin{prop} 
Suppose $G$ is a closed $\kappa$-Baire subgroup of $S_{\kappa}$ and $\mathcal{F}$ is the family of orbits of elements of $\kappa^{<\kappa}$. Then a subset of $X_{\tau}$ is $\kappa$-Borel and $G$-invariant if and only if it is definable in $\La_{\kappa^+\kappa}(\mathcal {F})$. 
\end{prop} 

\begin{proof} 
As in the proof of Proposition 9. 
\end{proof} 

Good quantifiers are defined as in section 2.2 but finite tuples are replaced by elements of $\kappa^{<\kappa}$. 

\begin{prop} 
Suppose $G$ is a closed $\omega_1$-Baire subgroup of $S_{\omega_1}$. There is a good binary quantifier $Q_G$ with $G=\Aut(Q_G)$. 
\end{prop} 

\begin{proof} Suppose $f:\omega_1\rightarrow\mathcal{P}(\omega)$ is injective. The proof is as the proof of Proposition 13, except that $P$ is replaced by the downward closure of the union of the orbits of 
$$\{\langle0,0\rangle\}\cup\{\langle n,n+1\rangle:n<\omega\}\cup\{\langle n,\alpha\rangle:\omega\leq\alpha<\gamma, n\in f(\alpha)\}$$ 
for $\gamma<\omega_1$. 
\end{proof} 

Moreover if $Q$ is a good quantifier on $\omega_1$, then a subset of $X_\tau$ is $\omega_1$-Borel and $G$-invariant if and only if it is definable in $\La_{\omega_2\omega_1}(Q)$. 

\begin{prop} 
The orbits of the automorphism group of a clopen quantifier $Q$ on $\kappa$ are definable in $\La_{\kappa\kappa}(Q)$. 
\end{prop} 

\begin{proof} 
As in the proof of Proposition 18. 
\end{proof}

\bibliography{refs}
\bibliographystyle{asl}

\end{document}